\documentclass[11pt,a4paper]{amsart}
\usepackage{amsmath,amsthm,amsfonts,graphicx}
\usepackage[numbers,sort&compress]{natbib}
\usepackage[lmargin=33mm,rmargin=33mm,tmargin=33mm,bmargin=33mm]{geometry}
\renewcommand{\baselinestretch}{1.25}
\usepackage[colorlinks=true,citecolor=black,linkcolor=black,urlcolor=blue]{hyperref}
\newcommand{\doi}[1]{\href{http://dx.doi.org/#1}{\texttt{doi:#1}}}
\newcommand{\arXiv}[1]{\href{http://arxiv.org/abs/#1}{\texttt{arXiv:#1}}}
\newcommand{\urlprefix}{}

\theoremstyle{plain}
\newtheorem{theorem}{Theorem}
\newtheorem{lemma}[theorem]{Lemma}

\newtheorem{proposition}[theorem]{Proposition}

\theoremstyle{definition}

\begin{document}

\title[Decompositions of Complete Multipartite Graphs into Complete
  Graphs]{Decompositions of Complete Multipartite Graphs\\ into Complete
  Graphs}

\author{Ruy Fabila-Monroy} \address{\newline Departamento de
  Matem\'aticas \newline Centro de Investigaci\'on y de Estudios
  Avanzados del Instituto Polit\'ecnico Nacional (Cinvestav) \newline
  Distrito Federal, M\'exico} \email{ruyfabila@math.cinvestav.edu.mx}

\author{David~R.~Wood} \address{\newline Department of Mathematics and
  Statistics \newline The University of Melbourne \newline Melbourne,
  Australia} \email{woodd@unimelb.edu.au}

\thanks{\textbf{MSC Classification}: 05B15 Orthogonal arrays, Latin
  squares, Room squares; 05C51 Graph designs and isomomorphic
  decomposition.}

\thanks{R.F.-M. is supported by an Endeavour Fellowship from the
  Department of Education, Employment and Workplace Relations of the
  Australian Government. D.W. is supported by a QEII Research
  Fellowship and a Discovery Project from the Australian Research Council.}

\date{\today}

\begin{abstract}
  Let $k\geq\ell\geq1$ and $n\geq 1$ be integers.  Let $G(k,n)$ be the
  complete $k$-partite graph with $n$ vertices in each colour
  class. An \emph{$\ell$-decomposition} of $G(k,n)$ is a set $X$ of
  copies of $K_k$ in $G(k,n)$ such that each copy of $K_\ell$ in
  $G(k,n)$ is a subgraph of exactly one copy of $K_k$ in $X$. This
  paper asks: when does $G(k,n)$ have an $\ell$-decomposition?  The
  answer is well known for the $\ell=2$ case. In particular, $G(k,n)$
  has a 2-decomposition if and only if there exists $k-2$ mutually
  orthogonal Latin squares of order $n$.  For general $\ell$, we prove
  that $G(k,n)$ has an $\ell$-decomposition if and only if there are
  $k-\ell$ Latin cubes of dimension $\ell$ and order $n$, with an
  additional property that we call mutually invertible. This property
  is stronger than being mutually orthogonal. An $\ell$-decomposition
  of $G(k,n)$ is then constructed whenever no prime less than $k$
  divides $n$.
\end{abstract}

\maketitle

\section{Introduction}

Let $G(k,n)$ be the complete $k$-partite graph with $n$ vertices in
each colour class. Formally, $G(k,n)$ has vertex set $[k]\times[n]$
where $(c,u)$ is adjacent to $(d,v)$ if and only if $c\neq d$. Here
$[n]:=\{1,2,\dots,n\}$. Sometimes we use a vector $(v_1,\dots,v_k)$ to
denote the clique with vertex set $\{(i,v_i):i\in[k]\}$.

For $k\geq\ell\geq2$, an \emph{$\ell$-decomposition} of $G(k,n)$ is a
set $X$ of copies of $K_k$ in $G(k,n)$, such that each copy of
$K_\ell$ in $G(k,n)$ is a subgraph of exactly one copy of $K_k$ in
$X$. Here $K_k$ is the complete graph on $k$ vertices. This paper
considers the question:

\emph{When does $G(k,n)$ have an $\ell$-decomposition?}

First note that every $\ell$-decomposition of $G(k,n)$ contains
exactly $n^\ell$ copies of $K_k$ (since $K_k$ contains
$\binom{k}{\ell}$ copies of $K_\ell$, and $G(k,n)$ contains
$\binom{k}{\ell}n^\ell$ copies of $K_\ell$).

The $\ell=2$ case of our question corresponds to a proper partition of
the edge-set of $G(k,n)$, called a `decomposition'. It is well known
that this case can be answered in terms of the existence of mutually
orthogonal Latin squares (Theorem~\ref{thm:Squares}). These
connections are explored in Section~\ref{sec:LatinSquares}.

Given this relationship, it is natural to consider the relationship
between $\ell$-decompositions and mutually orthogonal Latin cubes,
which are a higher dimensional analogue of Latin squares. However, the
situation is not as simple as the $\ell=2$ case. The first
contribution of this paper is a characterisation of
$\ell$-decompositions in terms of Latin cubes of dimension $\ell$,
with an additional property that we call mutually invertible
(Theorem~\ref{thm:CubeEquivalence}). This property is stronger than
being mutually orthogonal. For $\ell=2$ these two properties are
equivalent. These results are presented in
Section~\ref{sec:LatinCubes}.

Then in Section~\ref{sec:Construction}, we construct an
$\ell$-decomposition whenever no prime less than $k$ divides $n$
(Theorem~\ref{thm:NoSmallPrime}). Finally we relax the definition of
$\ell$-decomposition to allow each $K_\ell$ to appear in \emph{at
  least} one copy of $K_k$. Results are obtained for all $n$
(Theorem~\ref{thm:General}).

\section{Latin Squares and the $\ell=2$ Case}
\label{sec:LatinSquares}

A \emph{Latin square} of order $n$ is an $n\times n$ array in which
each row and each column is a permutation of $[n]$. Two Latin squares
are \emph{orthogonal} if superimposing them produces each element of
$[n]\times[n]$ exactly once. Two or more Latin squares are mutually
orthogonal (MOLS) if each pair is orthogonal.  If $L_1,\dots,L_{k-2}$
are mutually orthogonal Latin squares of order $n$, then it is easily
verified that the $n^2$ copies of $K_k$ defined by the vectors
\begin{equation}
  \label{eqn:ConstructDecomposition}
  (L_1(x,y),\dots,L_{k-2}(x,y),x,y)\enspace,
\end{equation}
where $(x,y) \in [n]^2$, form an edge-partition of $G(k,n)$.  In fact,
the following well-known converse result holds; see
\citep[page~162]{MOLS}.

\begin{theorem}
  \label{thm:Squares}
  $G(k,n)$ has a 2-decomposition if and only if there exists $k-2$
  mutually orthogonal Latin squares of order $n$.
\end{theorem}

There are at most $n-1$ MOLS of order $n$; see
\citep[page~162]{MOLS}. On the other hand, \citet{MacNeish} proved
that if $p$ is the least prime factor of $n$ then there exists $p-1$
MOLS of order $n$. With Theorem~\ref{thm:Squares} this implies:

\begin{proposition}
  If $p$ is the least prime factor of $n$ and $k=p+1$, then there
  exists an edge-partition of $G(k,n)$ into $n^2$ copies of $K_k$.
\end{proposition}

Bose, Shrikhande and Parker~\citep{BS-TAMS60,BSP-CJM60} proved that
for all $n$ except $2$ and $6$ there exists a pair of MOLS of order
$n$. With Theorem~\ref{thm:Squares} this implies:

\begin{proposition}
  For all $n$ except $2$ and $6$ there is an edge-partition of
  $G(4,n)$ into $n^2$ copies of $K_4$.
\end{proposition}

Other values of $k$ and $n$ for which there is a 2-decomposition of
$G(k,n)$ are immediately obtained by applying
Theorem~\ref{thm:Squares} with known results about the existence of
MOLS; see \citep{MOLS}.

\section{Latin Cubes}
\label{sec:LatinCubes}

A \emph{$d$-dimensional Latin cube of order $n$} is a function
$L:[n]^d\rightarrow[n]$ such that each row is a permutation of $[n]$;
that is, for all $i\in[d]$ and
$x_1,\dots,x_{i-1},x_{i+1},\dots,x_d\in[n]$,
$$\{L(x_1,\dots,x_{i-1},j,x_{i+1}):j\in[n]\}=[n]\enspace.$$

If $L_1,\dots,L_d$ are $d$-dimensional Latin cubes of order $n$, and
for every $(v_1,\dots,v_d)\in[n]^d$ there exists $x_1,\dots,x_d$ such
that $L_i(x_1,\dots,x_d)=v_i$ for all $i\in[d]$, then $L_1,\dots,L_d$
are said to be \emph{orthogonal}. Thus superimposing $L_1,\dots,L_d$ produces
each element of $[n]^d$ exactly once.  If every $d$-tuple of a set
$\mathcal{L}$ of $d$-dimensional Latin cubes of order $n$ are
orthogonal then $\mathcal{L}$ is \emph{mutually orthogonal}.  For
results on mutually orthogonal Latin cubes and related concepts see
\citep{Kerr-FQ82,Aggarwal75,Trenkler-CMJ05,ASS82,AS-FQ81,AHVS76,AS74,SS-AC08}.

From an $\ell$-decomposition of $G(k,n)$, it is possible to construct
a set of $k-\ell$ mutually orthogonal $\ell$-dimensional Latin cubes (see
Theorem~\ref{thm:CubeEquivalence}). However, the natural analogue of
\eqref{eqn:ConstructDecomposition} does not hold.  Consider the
following set $\{L_1,L_2,L_3\}$ of three mutually orthogonal
$3$-dimensional Latin cubes of order $4$.

\bigskip\noindent
\setlength{\tabcolsep}{1.25mm}
 \begin{tabular}{|c|c|c|c|}
    \hline
    $111$ & $233$ & $344$ & $422$\\ \hline
    $343$ & $421$ & $112$ & $234$\\ \hline 
    $424$ & $342$ & $231$ & $113$\\ \hline
    $232$ & $114$ & $423$ & $341$\\ \hline
  \end{tabular}
 \hspace*{1mm}
  \begin{tabular}{|c|c|c|c|}
    \hline
    $222$ & $\underline{14}4$ & $433$ & $311$\\ \hline
    $434$ & $312$ & $221$ & $\underline{14}3$\\ \hline
    $313$ & $431$ & $\underline{14}2$ & $224$\\ \hline
    $\underline{14}1$ & $223$ & $314$ & $432$\\ \hline
  \end{tabular}
 \hspace*{1mm}
  \begin{tabular}{|c|c|c|c|}
    \hline
    $333$ & $411$ & $122$ & $244$\\ \hline
    $121$ & $243$ & $334$ & $412$\\ \hline
    $242$ & $124$ & $413$ & $331$\\ \hline
    $414$ & $332$ & $241$ & $123$\\ \hline
  \end{tabular}
 \hspace*{1mm}
  \begin{tabular}{|c|c|c|c|}
    \hline
    $444$ & $322$ & $211$ & $133$\\ \hline
    $212$ & $134$ & $443$ & $321$\\ \hline
    $131$ & $213$ & $324$ & $442$\\ \hline
    $323$ & $441$ & $132$ & $214$\\ \hline
  \end{tabular}

\bigskip\noindent In this example, the Latin cubes are superimposed so that
$L_1$ is:
\begin{center}
  \medskip
  \begin{tabular}{|c|c|c|c|}
    \hline
    $1$ & $2$ & $3$ & $4$\\ \hline
    $3$ & $4$ & $1$ & $2$\\ \hline 
    $4$ & $3$ & $2$ & $1$\\ \hline
    $2$ & $1$ & $4$ & $3$\\ \hline
  \end{tabular}
 \hspace*{1mm}
  \begin{tabular}{|c|c|c|c|}
    \hline
    $2$ & $1$ & $4$ & $3$\\ \hline
    $4$ & $3$ & $2$ & $1$\\ \hline
    $3$ & $4$ & $1$ & $2$\\ \hline
    $1$ & $2$ & $3$ & $4$\\ \hline
  \end{tabular}
 \hspace*{1mm}
  \begin{tabular}{|c|c|c|c|}
    \hline
    $3$ & $4$ & $1$ & $2$\\ \hline
    $1$ & $2$ & $3$ & $4$\\ \hline
    $2$ & $1$ & $4$ & $3$\\ \hline
    $4$ & $3$ & $2$ & $1$\\ \hline
  \end{tabular}
 \hspace*{1mm}
  \begin{tabular}{|c|c|c|c|}
    \hline
    $4$ & $3$ & $2$ & $1$\\ \hline
    $2$ & $1$ & $4$ & $3$\\ \hline
    $1$ & $2$ & $3$ & $4$\\ \hline
    $3$ & $4$ & $1$ & $2$\\ \hline
  \end{tabular}

  \medskip
\end{center}
The natural analogue of \eqref{eqn:ConstructDecomposition} would be
to construct copies of $K_6$ in $G(6,4)$ of the form
$$(L_1(x,y,z),L_2(x,y,z),L_3(x,y,z),x,y,z)\enspace,$$
where $x,y,z\in[4]$. However, in this case not every copy of $K_3$ in 
$G(6,4)$ is covered. For example, $\{(1,1),(2,2),(6,2)\}$ is not
covered (since $z=2$ and $L_1(x,y,2)=1$ implies $L_2(x,y,2)=4$, as shown by the underlined entries above).

Below we introduce a stronger condition than orthogonality so that
this construction does provide an $\ell$-decomposition.

We consider $k$-tuples in $[n]^k$ to be functions from $[k]$ to
$[n]$. So that for $t:=(t_1,\dots,t_k)$, we use the notation
$t(i)=t_i$. A set $X$ of $k$-tuples in $[n]^k$ is said to be
\emph{$\ell$-extendable} if for all indices $s_1 < s_2 < \cdots < s_\ell$
(where $s_i \in [k]$) and for every element $(x_1,\dots,x_\ell) \in[n]^\ell$, there exists a unique $t \in X$ such that $t(s_i)=x_i$ for
all $i\in[\ell]$.

\begin{lemma}
  \label{lem:deftolatin}
  Let $X$ be an $\ell$-extendable set of $k$-tuples in $[n]^k$, and let
  $s_1 < s_2 < \cdots < s_\ell$, where $s_i \in [k]$. Let $t$ be
  the unique $k$-tuple such that $t(s_i)=x_i$ for all $i\in[\ell]$. 
For every $j   \in [k]\setminus\{s_1,s_2,\dots,s_\ell\}$, let $L_j$ be the function
  defined by $L_j(x_1,\dots,x_\ell):=t(j)$.  Then $L_j$ is an
  $\ell$-dimensional Latin cube.
\end{lemma}

\begin{proof}
  Let $(x_1,\dots,x_\ell) \in [n]^\ell$ and $h \in [\ell]$. Suppose
  that for some $x_h'\in [n]$,
$$L_j(x_1,\dots,x_{h-1},x_h,x_{h+1},\dots,x_\ell)=y=L_j(x_1,\dots,x_{h-1},x_h',x_{h+1},\dots,x_\ell)\enspace.$$
Then there is a tuple $t'$ in $X$ such that $t'(s_i)=x_i$ for $s_i \in
\{s_1,\dots,s_\ell\}\setminus \{s_h\}$ and $t(j)=y$.  Since $X$ is
$\ell$-extendable, this tuple is unique. Therefore $x_h=x_h'$ 
and $L_j$ is a Latin cube.
\end{proof}

A set $L_1,\dots,L_k$ of $\ell$-dimensional Latin cubes of order $n$
is said to be \emph{mutually invertible} if
$$\{(L_1(x_1,\dots,x_\ell),\dots,L_k(x_1,\dots,x_\ell),x_1,\dots,x_\ell) :(x_1,\dots,x_\ell)\in [n]^\ell\}$$
is $\ell$-extendable.

\begin{proposition}
  \label{prop:inv_ort}
  Every set $L_1,\dots,L_k$ of mutually invertible $\ell$-dimensional Latin cubes is
  mutually orthogonal.
\end{proposition}

\begin{proof}
  Let $s_1< s_2 <\cdots< s_\ell$ with $s_i \in [k]$ and let
  $(y_1,\dots,y_\ell)\in [n]^\ell$. It remains to show that there
  exists a unique $(x_1,\dots,x_\ell) \in [n]^d$ such that
  $$(L_{s_1}(x_1,\dots,x_\ell),\dots,L_{s_\ell}(x_1,\dots,x_\ell))=(y_1,\dots,y_\ell)\enspace.$$
  This follows from the fact that
$$\{(L_1(x_1,\dots,x_\ell),\dots,L_k(x_1,\dots,x_\ell),x_1,\dots,x_\ell) :(x_1,\dots,x_\ell)\in [n]^\ell\}$$
is $\ell$-extendable.
\end{proof}

In the case of $2$-dimensional Latin cubes, mutual orthogonality is
equivalent to mutual invertibility.

\begin{proposition}
  \label{prop:inv_squares}
  Every set $L_1,\dots,L_k$ of mutually orthogonal Latin squares is
  mutually invertible.
\end{proposition}

\begin{proof}
We prove that $$X:=\{(L_1(x,y),\dots,L_k(x,y),x,y):(x,y) \in [n]^2\}$$
is $2$-extendable.  Let $z_1,z_2\in[k+2]$ with $z_1< z_2$. We claim
that for each $(x_1,x_2) \in [n]$ there is a unique tuple $t \in
X $ such that $t(z_1)=x_1$ and $t(z_2)=x_2$. Consider the
following cases.
\begin{itemize}
\item $z_1=k+1$ and $z_2=k+2$: The claim immediately follows from
  the definition of $X$.
\item $z_1\leq k$ and $z_2\in\{k+1, k+2\}$: The claim  follows from the fact that $L_{z_1}$
  is a Latin square.
\item $z_1 \leq k$ and $z_2 \leq k$: The claim follows from the fact that $L_{z_1}$ and $L_{z_2}$
are orthogonal.
\end{itemize}
Therefore $X$ is $2$-extendable and $L_1,\dots,L_k$ is a set of mutually invertible
Latin squares.
\end{proof}

\begin{theorem}
  \label{thm:CubeEquivalence}
  $G(k,n)$ has an $\ell$-decomposition if and only if there are
  $k-\ell$ mutually invertible Latin $\ell$-dimensional cubes of order
  $n$.
\end{theorem}

\begin{proof}
  ($\Longleftarrow$) Let $L_1,\dots,L_{k-\ell}$ be $k-\ell$ mutually
  invertible $\ell$-dimensional Latin cubes of order $n$.  For each
  $(x_1,\dots,x_\ell)\in [n]^\ell$, let $K(x_1,\dots,x_\ell)$ be the
  copy of $K_k$ defined by the vector
  $(v_1,\dots,v_{k-\ell},x_1,\dots,x_\ell)$ where
  $v_i:=L_i(x_1,\dots,x_\ell)$. This defines $n^\ell$ copies of $K_k$.
That each copy of $K_\ell$ in $G(k,n)$ is in one such copy
  of $K_k$ follows immediately from the fact that
  $$\{(v_1,\dots,v_{k-\ell},x_1,\dots,x_\ell):(x_1,\dots,x_\ell) \in [n]^\ell\}$$
  is $\ell$-extendable.
 
  ($\Longrightarrow$) Consider an $\ell$-decomposition $X$ of
  $G(k,n)$. Thus $X$ is a set of copies of $K_k$ in $G(k,n)$ such that
  each copy of $K_\ell$ is in exactly one copy of $K_k$ in
  $X$. Consider each copy of $K_k$ in $X$ to be a $k$-tuple in
  $[n]^k$. We now show that $X$ is $\ell$-extendable.  Let $s_1<\cdots<s_\ell$ be
  elements of $[k]$ and $(x_1,\dots,x_\ell) \in [n]^\ell$. There is a
  unique tuple $(t_1,\dots,t_k)$ in $X$ containing the
copy of  $K_\ell$ with vertex set
  $\{(s_1,x_1),\dots,(s_\ell,x_\ell)\}$. Thus $t(s_i)=x_i$ for all
  $i\in[\ell]$. Therefore $X$ is $\ell$-extendable. By
  Lemma~\ref{lem:deftolatin}, we obtain $k-\ell$ mutually
  invertible Latin cubes.
\end{proof}

Note that Proposition~\ref{prop:inv_squares} and 
Theorem~\ref{thm:CubeEquivalence} provide a long-winded proof of
Theorem~\ref{thm:Squares}.

\section{Construction of an $\ell$-Decomposition}
\label{sec:Construction}

This section describes a construction of an $\ell$-decomposition.

\begin{lemma}
  \label{lem:Heart}
  If $n\geq k\geq\ell\geq 2$ and $n$ is prime, then $G(k,n)$ has an
  $\ell$-decomposition.
\end{lemma}

\begin{proof}
  Given $(a_1,\dots,a_\ell)\in[n]^\ell$, let $K(a_1,\dots,a_\ell)$ be
  the set of vertices
$$K(a_1,\dots,a_\ell):=\left\{\bigg(c,\Big(\sum_{j=0}^{\ell-1}c^ja_j\Big)\bmod{n}\bigg):c\in[k]\right\}$$
in $G(k,n)$. Observe that $K(a_1,\dots,a_\ell)$ induces a copy of
$K_k$ in $G(k,n)$, and we have $n^\ell$ such copies.  We claim that
each copy of $K_\ell$ is in some $K(a_1,\dots,a_\ell)$. Let
$S=\{(c_i,v_i):i\in[\ell]\}$ be a set of vertices inducing $K_\ell$. Thus $c_i\neq c_j$ for all $i\neq j$. We need to show that
$S\subseteq K(a_1,\dots,a_\ell)$ for some $a_1,\dots,a_\ell$. That is,
for all $i\in[\ell]$,
$$\sum_{j=0}^{\ell-1}c_i^ja_j\equiv v_i\pmod{n}\enspace.$$
Equivalently,
\begin{equation}
  \label{eqn:Matrix}
  \begin{bmatrix}
    1&c_1&c_1^2&\dots&c_1^{\ell-1}\\
    1&c_2&c_2^2&\dots&c_2^{\ell-1}\\
    &&&\vdots\\
    1&c_\ell&c_\ell^2&\dots&c_\ell^{\ell-1}
  \end{bmatrix}
  \begin{bmatrix}
    a_1\\
    a_2\\
    \vdots\\
    a_\ell
  \end{bmatrix}
  \equiv
  \begin{bmatrix}
    v_1\\
    v_2\\
    \vdots\\
    v_\ell
  \end{bmatrix}
  \pmod{n}\enspace.
\end{equation}
This $\ell\times\ell$ matrix is a Vandermonde matrix, which has
non-zero determinant $$\prod_{1\leq i<j\leq\ell}(c_i-c_j)\enspace.$$
Since $c_i\neq c_j$ and $n$ is a prime greater than any $c_i-c_j$,
this determinant is non-zero modulo $n$. (This trick of taking a
Vandermonde matrix modulo a prime is well known, and at least dates to
a 1951 construction by \citet{Erdos51} for the no-three-in-line
problem.)\ Thus in the vector space $\mathbb{Z}_n^\ell$ (over the
finite field $\mathbb{Z}_n$), the row-vectors of this matrix are
linearly independent and \eqref{eqn:Matrix} has a solution. That is,
$S\subseteq K(a_1,\dots,a_\ell)$ for some $a_1,\dots,a_\ell$.
\end{proof}

The next lemma is analogous to a Kronecker product of Latin squares.

\begin{lemma}
  \label{lem:BasicProduct}
  For all integers $k\geq\ell\geq 1$ and $p,q\geq1$, if both $G(k,p)$
  and $G(k,q)$ have $\ell$-decompositions, then $G(k,pq)$ has an
  $\ell$-decomposition.
\end{lemma}

\begin{proof}
  Let $X_1,\dots,X_{p^\ell}$ be the vertex sets of copies of $K_k$ in
  $G(k,p)$ such that each $K_\ell$ subgraph appears in exactly one
  copy. Similarly, let $Y_1,\dots,Y_{q^\ell}$ be the vertex sets of
  copies of $K_k$ in $G(k,q)$ such that each $K_\ell$ subgraph of
  $G(k,q)$ appears in exactly one copy.  For $a\in[p^\ell]$ and
  $b\in[q^\ell]$, if $X_a=\{(i,v_i):i\in[k]\}$ and
  $Y_b=\{(i,w_i):i\in[k]\}$, then let $Z_{a,b}$ be the set of vertices
  $\{(i,(w_i-1)p+v_i):i\in[k]\}$ in $G(k,pq)$. Thus $Z_{a,b}$ induces
  a copy of $K_k$.

  Let $S=\{(c_i,u_i):i\in[\ell]\}$ be a set of vertices inducing a
  $K_\ell$ in $G(k,pq)$. Say $u_i=(w_i-1)p+v_i$ where
  $v_i\in[p]$ and $w_i\in[q]$. Since $\{(c_i,v_i):i\in[k]\}$ induces
  $K_\ell$ in $G(k,p)$, some $K_a$ contains
  $\{(c_i,v_i):i\in[k]\}$. Similarly, some $K_b$ contains
  $\{(c_i,w_i):i\in[k]\}$. By construction, $S\subseteq
  Z_{a,b}$. Hence the $Z_{a,b}$ are the vertex sets of copies of $K_k$
  in $G(k,pq)$ such that each $K_\ell$ subgraph of $G(k,pq)$ appears
  in some copy. There are $(pq)^\ell$ such sets $Z_{a,b}$. Thus the
  $Z_{a,b}$ are an $\ell$-decomposition of $G(k,pq)$.
\end{proof}

Lemmas~\ref{lem:Heart} and \ref{lem:BasicProduct} imply the following,
which is one of the main results of the paper.

\begin{theorem}
  \label{thm:NoSmallPrime}
  If $n\geq k\geq\ell\geq 2$ and no prime less than $k$ divides $n$,
  then $G(k,n)$ has an $\ell$-decomposition.
\end{theorem}

Theorems~\ref{thm:CubeEquivalence} and \ref{thm:NoSmallPrime} imply:

\begin{theorem}
  If $n\geq k\geq\ell\geq 2$ and no prime less than $k$ divides $n$,
  then there exists a set of $k-\ell$ mutually invertible
  $\ell$-dimensional Latin cubes.
\end{theorem}

To generalise the above results, consider the following definition.
For integers $k\geq\ell\geq 1$ and $n\geq1$, let $f(k,n,\ell)$ be the
minimum number of copies of $K_k$ in $G(k,n)$ such that each $K_\ell$
subgraph of $G(k,n)$ appears in some copy. Note that $f(k,n,\ell)\geq
n^{\ell}$ because no two of the $n^{\ell}$ copies of $K_\ell$ that are
contained in the first $\ell$ colours classes of $G(k,n)$ are
contained in a single copy of $K_k$. And $f(k,n,\ell)=n^\ell$ if and
only if $G(k,n)$ has an $\ell$-decomposition.

\begin{lemma}
  \label{lem:PrimeLemma} For all $n$ and all $k$, there is an integer
  $n'$ such that $n\leq n'\leq n+e^{k+o(k)}$ and no prime less than
  $k$ divides $n'$.
\end{lemma}

\begin{proof}
  Let $p$ be the product of all primes less than $k$. Let $n'$ be the
  minimum integer such that $n'\geq n$ and $n'\equiv 1\pmod{p}$. Thus
  $n'\leq n+p$ and no prime less than $k$ divides $n'$.  By the
  asymptotics of primorials, $p\leq e^{k+o(k)}$; see
  \citep{Primorials}. The result follows.
\end{proof}

Theorem~\ref{thm:NoSmallPrime} and Lemma~\ref{lem:PrimeLemma} imply
that $f(k,n,\ell)$ is never much more than $n^\ell$.

\begin{theorem}
  \label{thm:General}
  For fixed $k\geq\ell\geq 1$ and $n\geq1$,
$$f(k,n,\ell)\leq n^\ell+O(n^{\ell-1})\enspace.$$
\end{theorem}

Finally we mention that Theorem~\ref{thm:General} with $k=6$ and
$\ell=3$ was recently applied to a problem in combinatorial geometry
\citep{FW-Triangles}. Indeed, this problem instigated our research.

\section{Note}

After submitting this paper we discovered that much of it is known in
the literature on ``orthogonal arrays'' and ``covering arrays''.  An
$\ell$-extendable set of $k$-tuples in $[n]^k$ is the set of columns
of an orthogonal array with $k$ constraints, $n$ levels and strength
$\ell$ (see \citep{OrthogonalArrays}), and $f(k,n,\ell)$ is the
covering array number $\operatorname{CAN}(\ell,k,n)$ (see
\citep{CoveringArrays}).  See \citep{BoseBush,Bush52a,Bush52b} for
some of the seminal results on orthogonal arrays, and see
\citep{Hartman,CoveringArrays} for more recent surveys. 
Our Lemma~\ref{lem:Heart} is Theorem~3.2 in \citep{Hartman}, our
Lemma~\ref{lem:BasicProduct} is Theorem~3.4 in \citep{Hartman}, and
our Theorem~\ref{thm:NoSmallPrime} is Corollary~3.5 in
\citep{Hartman}. Other results in this paper are probably previously
known.


\def\soft#1{\leavevmode\setbox0=\hbox{h}\dimen7=\ht0\advance \dimen7
  by-1ex\relax\if t#1\relax\rlap{\raise.6\dimen7
  \hbox{\kern.3ex\char'47}}#1\relax\else\if T#1\relax
  \rlap{\raise.5\dimen7\hbox{\kern1.3ex\char'47}}#1\relax \else\if
  d#1\relax\rlap{\raise.5\dimen7\hbox{\kern.9ex \char'47}}#1\relax\else\if
  D#1\relax\rlap{\raise.5\dimen7 \hbox{\kern1.4ex\char'47}}#1\relax\else\if
  l#1\relax \rlap{\raise.5\dimen7\hbox{\kern.4ex\char'47}}#1\relax \else\if
  L#1\relax\rlap{\raise.5\dimen7\hbox{\kern.7ex
  \char'47}}#1\relax\else\message{accent \string\soft \space #1 not
  defined!}#1\relax\fi\fi\fi\fi\fi\fi}

\end{document}